\documentclass[12pt]{amsart}
\usepackage{amssymb}
\usepackage{verbatim}
\usepackage{amscd}

\input{xy}
\xyoption{all}

\theoremstyle{plain}
\newtheorem{theorem}{Theorem}
\newtheorem{corollary}[theorem]{Corollary}

\newtheorem{lemma}[theorem]{Lemma}

\theoremstyle{definition}

\theoremstyle{remark}
\newtheorem{remark}{Remark}

\newcommand{\cstar}{\ensuremath{\text{C}^{*}}\nobreakdash-\hspace{0 pt}}

\newcommand{\ds}{\displaystyle}

\newcommand{\ol}[1]{\overline{#1}}

\newcommand{\scD}{\sA(\bbD)\times_{\alpha}\bbZ^+}
\newcommand{\scC}{C(\bbT)\times_{\alpha}\bbZ^+}

\def\bbD{\mathbb D}

\def\bbT{\mathbb T}
\def\bbZ{\mathbb Z}

     \newcommand{\sA}{\mathcal A}
     
     \newcommand{\sC}{\mathcal C}

     \newcommand{\sJ}{\mathcal J}

     \newcommand{\sN}{\mathcal N}
     \newcommand{\sO}{\mathcal O}
     \newcommand{\sP}{\mathcal P}

\newcommand{\al}{\alpha}
\newcommand{\be}{\beta}

\newcommand{\vpi}{\varphi}
\newcommand{\si}{\sigma}

\newcommand{\ga}{\gamma}

\begin{document}

\title[Semicrossed products of the disk algebra]{Semicrossed products of the disk algebra and the Jacobson radical}

\author[Khemphet]{Anchalee Khemphet}
\address{Department of Mathematics\\
Iowa State University\\
Ames, Iowa\\
USA}
 \email{khemphet@iastate.edu}

\author[Peters]{Justin R.\ Peters $\ddag$}
\address{Department of Mathematics \\
Iowa State University \\
Ames, Iowa\\
USA} \email{peters@iastate.edu}

\begin{abstract} We consider semicrossed products of the disk algebra with respect to
endomorphisms defined by finite Blaschke products. We characterize the Jacobson radical
of these operator algebras. Furthermore, in the case the finite Blaschke product is elliptic,
 we show that the semicrossed  product contains no nonzero quasinilpotent
elements. However, if the finite Blaschke product is hyperbolic or parabolic with zero hyperbolic step,
the Jacobson radical is nonzero and a proper subset of the set of quasinilpotent elements.
\end{abstract}


 \thanks{{{$\ddag$} The author acknowledges partial support from the National Science Foundation, DMS-0750986}}

\maketitle

\section{Introduction}

While semicrossed products have studied for the past thirty years or more, in the past decade there has been a resurgence of interest in this class of nonselfadjoint operator algebras.  One is given an
endomorphism $\al$ of a C$^*$-algebra $\sC$  and one constructs the semicrossed product $\sC\times_{\al}\bbZ^+.$   ( e.g.,\cite{DK3},  \cite{K})
On the other hand, until recently little had been done with semicrossed products of other classes of operator algebras, though now the literature is growing (e.g., \cite{DK1}, \cite{DK2}, \cite{KK}).

In this note we consider semicrossed products of the disk algebra.  The case in which the endomorphism of the disk algebra is an automorphism, and hence implemented by a conformal map of the disc, was considered in \cite{HPW} and \cite{BP}. In a recent paper of Davidson and Katsoulis, they consider semicrossed products of the disk algebra by endomorphisms associated with finite Blaschke products. Of course, such an endomorphism of the disk algebra $\sA(\bbD)$  can also be viewed as an endomorphism of $C(\bbT),$ the continuous functions on the unit circle. They show that the semicrossed product of the disk algebra is naturally a subalgebra of the semicrossed product of $C(\bbT),$ and both algebras have the same C$^*$-envelope.

Donsig, Katalvolos and Manoussos provided a description of the Jacobson radical of semicrossed product of $C(X)$ where $X$ is a compact metric space with a map $\vpi: X \to X$ a continuous
surjection, so that $f \in C(X) \to f \circ \vpi $ is an endomorphism of $C(X).$ We show that their description of the radical applies equally well to semicrossed products of the disk algebra
with respect to finite Blaschke products (Theorem~\ref{t:jacobsonradical}). On the other hand, for semicrossed products of $C(\bbT)$ there will always be an abundance of quasinilpotent elements, indeed
nilpotent elements, even if the algebra is semi-simple. 

In the case the Jacobson radical is nonzero, we can show that the set of quasinilpotent elements properly contains the Jacobson radical.  However, if the endorphism is implemented by an
 elliptic finite Blaschke product, we show that the set of quasinilpotent elements coincides with the Jacobson radical, which is $(0).$

\section{Semicrossed Products}

Let $\vpi$ be a (non-constant) finite Blaschke product, which is not a M\"{o}bius transformation. This gives an endomorphism $\al$ of the disk algebra, $f \in \sA(\bbD) \to \al(f) = f\circ \vpi$
which is not an automorphism. Consider the algebra $\sP$ of formal polynomials
\[ \sum U^k f_k  \]
with $f_k \in \sA(\bbD)$ where $f \in \sA(\bbD)$ and $U$ satisfy the commutation relation $ fU = U \al(f).$
Any pair $(A, T)$ of contractions satisfying $ AT = T\vpi(A)$ defines a representation of $\sP.$  First note there is a contractive representation $\rho$ of the disk algebra $\sA(\bbD)$ given
by $\rho(f) = f(A).$ Then set
\[ \rho \times T (\sum U^k f_k) = \sum T^k \rho(f_k) .\]
The norm on $\sP$  defined by taking the supremum of the norms of all such representations is an operator algebra norm. The completion of $\sP$ with respect to this norm is the
semicrossed product, $\scD.$ A good discussion of this algebra can be found in \cite{DK}.

There is a one-parameter automorphism group $ t \in \bbT \to \ga_t$ acting on the algebra $\sP$ by $\ga_t(\sum U^k f_k) = \sum U^k e^{2\pi i k t} f_k$ which extends to
an automorphism, which also denote $\ga_t$, of $\scD.$ It follows from the definition of the norm that $\ga_t$ is isometric.

Let $F \in \scD.$ Then we may define the $k^{th}$ Fourier coefficient of $F$, denoted by $\pi_k(F),$ by 
\[ U^k\pi_k(F) = \int_{\bbT} e^{-2\pi i k t} \ga_t(F) \, dm(t) \]
where $dm$ is normalized Lebesgue measure on the torus.  In case $F = \sum U^j f_j$ is in the polynomial algebra $\sP,$ then $\pi_k(F) = f_k$, which is $0$ if there is no $U^k$ term in the 
polynomial.  From this it follows that $||\pi_k(F)|| \leq ||F||.$ We refer to $\pi_k(F) $ as the $k^{th}$ \textit{ Fourier coefficient} of $F.$

One result we will need, which in particular follows from \cite{DK} and \cite{K}, is that for any $f \in \sA(\bbD)$ and any $k = 0, 1, \dots, \ ||U^k f|| =||f||$ where by $||f|| $ we mean, of course,
the norm of $f$ in the disk algebra.  This result follows from the fact that there is a faithful representation of the semicrossed product which maps $U$ to an isometry.

Now the endomorphism $\al$ extends to an endomorphism, also denoted by $\al$, of the continuous functions on the circle, $C(\bbT).$ 
 When we compute the radical of the semicrossed product, we make use of
the following theorem, which is Theorem 2.1 of \cite{DK}.
\begin{theorem} Let $\vpi$ be a finite Blaschke product, and let $\al(f) = f \circ \vpi.$  Then $\scD$ is (canonically completely isometrically isomorphic to) a subalgebra of $\scC.$
\end{theorem}
It is also shown in \cite{DK} that these two algebras have the same $C^*$-envelope (\cite{DK} Corollary 2.4), though we will not need that result.

Furthermore, let $\al, \be$ be endomorphisms of $\sA(\bbD)$ given by $\al(f) = f\circ\vpi$, $\be(f) = f\circ\psi$ where $\psi = \tau^{-1}\circ\vpi\circ\tau$ and $\tau$ is a conformal map of $\bbD$. We have the following result as in the case of automorphisms (see \cite{HPW}).

\begin{theorem} \label{t:conjugate}
Define $\Phi:\scD\rightarrow\sA(\bbD)\times_{\be}\bbZ^+$ on the dense subalgebra of polynomials by $\Phi(\sum U_{\al}^kf_k) = \sum U_{\be}^kf_k\circ\tau$.
 Then $\Phi$ is an isometric isomorphism on the subalgebra, and hence extends to a complete isometric isomorphism of the semicrossed products.
\end{theorem}
\begin{proof}
The proof that shows $\Phi$ is an algebra homomorphism is similar to \cite{HPW} Lemma 10. Thus, we need only show that $\Phi$ is isometric.
Let $(A, T)$ be a pair of contractions satisfying $AT = T\vpi(A)$. Then there is a contractive representation $\rho$ of $\sA(\bbD)$ such that $\rho(f) = f(A)$. 
Notice that $\tau^{-1}(A)T = T\tau^{-1}(\vpi(A)) = T\psi(\tau^{-1}(A))$. Then there is a contractive representation $\si$ of $\sA(\bbD)$ such that $\si(f) = f(\tau^{-1}(A))$.
Therefore,
\begin{align*}
\si\times T(\Phi(\sum U_{\al}^kf_k)) &= \si\times T(\sum U_{\be}^kf_k\circ\tau)\\
&= \sum T^k\si(f_k\circ\tau)\\
&= \sum T^kf_k(A)\\
&= \sum T^k\rho(f_k)\\
&= \rho\times T(\sum U_{\al}^kf_k).
\end{align*}
The same argument shows that for each covariant representation $\si\times T$ of $\sA(\bbD)\times_{\be}\bbZ^+$, we can find a corresponding covariant representation $\rho\times T$ of $\scD$ so that the above equation holds. Hence, $||\sum U_{\al}^kf_k|| = ||\sum U_{\beta}^kf_k\circ\tau||,$ and furthermore this is a complete isometric isomorphism.
\end{proof}

\section{Finite Blaschke Products}

Let $\mathbb{D}$ be the open unit disc and let $\mathbb{T}$ be the unit circle.
A finite Blaschke product $\varphi$ is a rational function of the form
\begin{equation}
\varphi(z) = u~\displaystyle\prod^{N}_{i=1}\frac{z-a_i}{1-\bar{a}_iz} \label{e:1}
\end{equation}
where $|u|=1$ and $|a_i| < 1$ for $i=1, 2, \dots, N$. In this paper, we assume that all functions $\varphi$ are finite Blaschke products.

Any finite Blaschke product is an analytic self-mapping of $\mathbb{D}$ and maps its boundary $\mathbb{T}$ onto itself. 
Moreover, it is an $N$-to-one function which preserves orientation. 
If $N=1$, then $\varphi$ is a conformal mapping. We say that $\varphi$ is non-trivial if $N>1$.
Any non-trivial finite Blaschke product is a local homeomorphism. The finite Blaschke products here will be non-trivial.

By the Schwarz-Pick Lemma, any finite Blaschke product $\varphi$ can have at most one fixed point in $\mathbb{D}$. If such a point exists, it must be an attracting fixed point and we call it the Denjoy-Wolff point. If $\varphi$ is a finite Blaschke product with $N$ factors and has a fixed point in $\mathbb{D}$, then $\varphi$ has at most $N-1$ fixed points on $\mathbb{T}$.

If $\varphi$ has no fixed points in $\mathbb{D}$, the Denjoy-Wolff Theorem states that there is a fixed point $z_0\in\mathbb{T}$ such that $0 < \varphi'(z_0)\leq 1$. This point is again called the Denjoy-Wolff point. Conversely, if $\varphi$ has a fixed point $z_0\in\mathbb{T}$ such that $0 < \varphi'(z_0)\leq 1$, then $\varphi$ has no fixed point in $\mathbb{D}$, and $z_0$ is the Denjoy-Wolff point of $\varphi$. Moreover, the proof of the Denjoy-Wolff Theorem guarantees that $\varphi$ can have at most one fixed point on $\mathbb{T}$ with $0 < \varphi'(z_0)\leq 1$. Therefore, the Denjoy-Wolff point is unique (see \cite{Shapiro} p.78). Thus, a finite Blaschke product $\varphi$ can be classified into three categories based on the Denjoy-Wolff point.

A finite Blaschke product $\varphi$ is said to be
\begin{itemize}
	\item[1.] elliptic if there exists a fixed point in $\mathbb{D}$;
	\item[2.] hyperbolic if the Denjoy-Wolff point $z_0\in\mathbb{T}$ is such that \mbox{$\varphi '(z_0) < 1$;}
	\item[3.] parabolic if  the Denjoy-Wolff point $z_0\in\mathbb{T}$ is such that \mbox{$\varphi '(z_0) = 1$.}
\end{itemize}

Denote the composition of $\varphi$ with itself $n$ times by $\varphi^n$. A point $z$ is said to be a recurrent point of $\varphi$ if there exists a sequence $\{n_k\}$ such that $\varphi^{n_k}(z)$ converges to $z$. Note that the Julia set is a subset of the closure of the set of recurrent points since the Julia set is the closure of the repelling periodic points. We know that the Julia set $\mathcal{J}$ of a finite Blaschke product is either $\mathbb{T}$ or a Cantor set of $\mathbb{T}$ (see \cite{Carleson-Gamelin} Example p.58).

Here we are interested in recurrent points in $\mathbb{T}$. If the Julia set of a finite Blaschke product is the unit circle, then the closure of the set of recurrent points of a finite Blaschke product is exactly the Julia set. Thus, recurrent points are dense in $\mathbb{T}$. \cite{BCH}, \cite{CDP} and \cite{Hamilton} give  necessary and sufficient conditions for the Julia set of a finite Blaschke product to equal  $\mathbb{T}$. First, we need the following definitions.

The hyperbolic distance $d$ in $\mathbb{D}$ is defined by
\begin{equation*}
d(z, w) = \log\frac{1+\frac{|z-w|}{|1-\overline{z}w|}}{1-\frac{|z-w|}{|1-\overline{z}w|}}
\end{equation*}
where $z, w\in{\mathbb{D}}$. A finite Blaschke product $\varphi$ is of \textit{zero hyperbolic step} if there exists $z\in\mathbb{D}$ such that $\displaystyle\lim_{n\rightarrow\infty}d(\varphi^n(z), \varphi^{n+1}(z)) = 0$, where $d$ is the hyperbolic distance in $\mathbb{D}$. If $\varphi$ is not of zero hyperbolic step, then 
$\displaystyle\lim_{n\rightarrow\infty}d(\varphi^n(z), \varphi^{n+1}(z)) > 0$ since $d(\varphi^n(z), \varphi^{n+1}(z))$ is non-increasing. In this case, we say $\varphi$ is of \textit{positive hyperbolic step.}

Let $(\mathbb{T}, \varphi, m)$ be a measure space where $m$ is normalized Lebesgue measure on $\mathbb{T}$. Then $\varphi$ is said to be ergodic if, whenever
 $\varphi^{-1}(E) = E$ up to measure zero, implies $m(E) = 0$ or $1$. Note that $\varphi$ is defined to be ergodic without assuming that $\varphi$ is measure-preserving,
 that is, $m(\varphi^{-1}(E)) = m(E)$ for any measurable set $E$.

\begin{theorem}[see \cite{BCH}, \cite{CDP} and \cite{Hamilton}] \label{t:juliasetT}
A finite Blaschke product $\varphi$ is of zero hyperbolic step if and only if its Julia set is exactly $\mathbb{T}$ if and only if it is $ergodic$.
\end{theorem}

Moreover, \cite{BCH} and  \cite{CDP} explain when $\varphi$ is of zero or positive hyperbolic step according to the classification based on the Denjoy-Wolff point. 

\begin{theorem}[see \cite{BCH}] \label{t:bch1}
Let $\varphi$ be a finite Blaschke product.
\begin{enumerate}
	\item If $\varphi$ is elliptic, then it is of zero hyperbolic step.
	\item If $\varphi$ is hyperbolic, then it is of positive hyperbolic step.
\end{enumerate}
\end{theorem}

\begin{theorem}[see \cite{CDP}] \label{t:cdp1}
If $\varphi$ is a parabolic finite Blaschke product with Denjoy-Wolff point $1$, then $\varphi$ is of zero hyperbolic step if and only if $\varphi''(1) = 0$ if and only if $\displaystyle\sum^N_{i=1}\frac{1-|a_i|^2}{|1-a_i|^2} \text{Im}(a_i) = 0$ where $a_i$'s are as in $(\ref{e:1})$.
\end{theorem}

Thus, we have the following lemma.
\begin{lemma} \label{l:zerostep}
If $\varphi$ is either elliptic or parabolic with zero hyperbolic step, then 
\begin{enumerate}
	\item the recurrent points of $\varphi$ are dense in $\mathbb{T}$;
	\item for any open set $\sO\subset\mathbb{T}$, there is an $N\in\mathbb{N}$ such that $\varphi^N(\sO) = \mathbb{T}$.
\end{enumerate}
\end{lemma}
\begin{proof}
(1) is clearly from Theorem \ref{t:juliasetT}, \ref{t:bch1}, and  \ref{t:cdp1}. Also, since the Julia set of $\varphi$ is $\mathbb{T}$, (2) is satisfied (see \cite {Carleson-Gamelin} Theorem 3.2).
\end{proof}

One might ask in the case where the Julia set of a finite Blaschke product is not the unit circle, namely, a Cantor set of $\mathbb{T}$, what is the closure of the set of recurrent points of a finite Blaschke product? By making use of the result from \cite{CDP}, we show that in fact it is equal to the union of the Julia set and the Denjoy-Wolff point. First, we provide the result needed to complete our proof.

\begin{theorem}[see \cite{CDP}] \label{t:cpd2}
Suppose that $\varphi$ is a finite Blaschke product.
If $\varphi$ is either  hyperbolic or parabolic with positive hyperbolic step, then $(\varphi^n(z))$ converges to the Denjoy-Wolff point of $\varphi$, for almost every $z\in\mathbb{T}$.
\end{theorem}

\begin{remark} \label{r:measjuliaset}
This theorem together with Theorem~\ref{t:juliasetT} imply that the normalized Lebesgue measure of the Julia set of finite Blaschke products is either 0 or 1. As a result, if the Julia set is a Cantor set, the Lebesgue measure of the Julia set is equal to zero.
\end{remark}

\begin{theorem} \label{t:recpoints}
If $\varphi$ is either hyperbolic or parabolic with positive hyperbolic step, then the closure of the set of recurrent points of a finite Blaschke product is the union of the Julia set and the Denjoy-Wolff point.
\end{theorem}
\begin{proof}
Let $\mathcal{J}$ be the Julia set and $\mathcal{F} = \overline{\mathbb{C}}\setminus\mathcal{J}$ the Fatou set of $\varphi$. Suppose $z_0$ is the Denjoy-Wolff point of $\varphi$. Then $z_0$ is either an attracting or rationally neutral fixed point. Let $A = \{z\in\mathbb{T} : \varphi^n(z)\rightarrow z_0\}$.
Since the basin of attraction of $z_0$ is open, $A$ is open in $\mathbb{T}$.
Note that $A\subset\mathcal{F}$ and $\mathcal{J}\subset\mathbb{T}\setminus A$.
By Theorem \ref{t:cpd2}, the Lebesgue measure of $\mathbb{T}\setminus A$ is zero.
Thus, any open set $O$ in $\mathbb{T}$ must contain a point $z\in A$.
Then, $A$ is dense in $\mathbb{T}$.
Therefore, $\mathbb{T} = \overline{A} = A\cup\partial A$.
Note that $\mathbb{D}$ and $\overline{\mathbb{C}}\setminus\overline{\mathbb{D}}$ are the subsets of $\mathcal{F}$. Since the Julia set is the boundary of a union of components of the Fatou set, $\mathcal{J} = \partial A = \mathbb{T}\setminus A$.
Next, let $R$ be the set of recurrent points.
Notice that $z_0$ is in $R$ and $\mathcal{J}\subset\overline{R}$.
Let $z\in R$. If $z\in A$, then $z = z_0$. Thus, $R\subset\mathcal{J}\cup\{z_0\}$.
Since $\mathcal{J}$ is closed, $\overline{R}\subset\mathcal{J}\cup\{z_0\}$.
Hence, $\overline{R} = \mathcal{J}\cup\{z_0\}$.
\end{proof}

Recall that $\varphi$ is said to be topologically transitive if there is some $z_0\in\mathbb{T}$ such that $O_{\varphi}(z_0) := \{\varphi^n(z_0)~|~n\in\mathbb{N}\}$ is dense in $\mathbb{T}$. Clearly, the recurrent points are dense in $\mathbb{T}$ if $\varphi$ is topologically transitive. However, the converse is not necessarily true. Consequently, this condition is stronger than the recurrent points  of $\varphi$ being dense. For this reason, we include the result in this paper though it is not needed to prove our main theorem.

\begin{theorem} \label{t:toptransitive}
Let $\varphi:\mathbb{T}\rightarrow\mathbb{T}$ be a finite Blaschke product. If $\varphi$ is elliptic or parabolic with zero hyperbolic step, then $\varphi$ is topologically transitive.
\begin{proof}
Let $V_n$ be a countable base of open arcs for $\mathbb{T}$. Let $U$ be any open subarc of $\mathbb{T}$.
By Lemma~\ref{l:zerostep}, $V_1\subset\varphi^{n_1}(U)$ for some $n_1\in\mathbb{N}$. We can find a subarc $U_1$ of $U$ such that $V_1 = \varphi^{n_1}(U_1)$. Let $W_1\subset\overline{W_1}\subset U_1$ be such that $\varphi^{n_1}|_{W_1}$ is one-to-one. Then $\varphi^{n_1}(\overline{W_1})\subset V_1$.
Repeat with $V_2$ in place of $V_1$ and $W_1$ in place of $U$. Then, $V_2\subset\varphi^{n_2}(W_1)$ for some $n_2\in\mathbb{N}$. Therefore, there is a subarc $U_2$ of $W_1$ such that $V_2 = \varphi^{n_2}(U_2)$. Let $W_2\subset\overline{W_2}\subset U_2$ be such that $\varphi^{n_2}|_{W_2}$ is one-to-one. Then $\varphi^{n_2}(\overline{W_2})\subset V_2$. 
Keep repeating this process, we have $\{W_n\}$ such that $W_{n+1}\subset W_n$. Thus, $\displaystyle\bigcap^{\infty}_{n=1}\overline{W_n}\neq\emptyset$. Moreover, $\varphi^{n_k}(\overline{W_k})\subset V_k$ for each $k\in\mathbb{N}$. That is, let $z_0\in\ds\bigcap^{\infty}_{n=1}\ol{W_n}$. Then $\varphi^{n_k}(z_0)\in V_k$ for each $k\in\mathbb{N}$. Therefore, $O_{\varphi}(z_0) = \{\varphi^n(z_0)~|~n\in\mathbb{N}\}$ is dense in $\mathbb{T}$.
Hence, $\vpi$ is topological transitive.
\end{proof}
\end{theorem}

\section{Quasinilpotent elements and the radical}

\begin{lemma} \label{l:quasinilpotent1}
\begin{enumerate}
\item Let $F\in A(\mathbb{D})\times_\alpha\mathbb{Z}^+$. If  the $0^{th}$ Fourier coefficient is not identically zero or
 $F$ has a Fourier coefficient which does not vanish on the fixed point set of a finite Blaschke product $\varphi$, then $F$ is not quasinilpotent. 
\item Let $f\in\mathcal{A}(\mathbb{D})$. If $f$ does not vanish on the set of recurrent points of $\varphi$, then there is an $n\in\mathbb{N}$ such that $U^{n} f$ is not quasinilpotent.
\end{enumerate}
\end{lemma}
\begin{proof}
The first part is a straightforward computation, which we omit.

For the second part, note that $f$ vanishes on the set of recurrent points if and only if $f$ vanishes on its closure. 
Assume that $\varphi$ is either hyperbolic or parabolic with positive hyperbolic step. From Theorem~\ref{t:recpoints}, the closure of the set of recurrent points is the union of the Julia set and the Denjoy-Wolff point.
By the first part, if $f$ does not vanish at the Denjoy-Wolff point, then $U^n f$ is not quasinilpotent for any $n$. 
Assume that $f$ vanishes on the Denjoy-Wolff point and that $f$ does not vanish on the Julia set.
Since repelling periodic points are dense in the Julia set, $f(x_0)\neq 0$ for some $x_0$ such that $\varphi^n(x_0) = x_0$ for some $n$.
Let $F = U^nf$. For any $k > 0$,
\begin{align*}
||F^k|| &\geq ||\pi_{nk}(F^k)||\\
&\geq |\pi_{nk}(F^k)(x_0)|\\
&= |f(x_0) f\circ\varphi^n(x_0)\dots f\circ\varphi^{(k-1)n}(x_0)|\\
&= |f(x_0)|^k.
\end{align*}
Therefore, $\displaystyle\lim_{k\rightarrow\infty}||F^k||^{\frac{1}{k}} \geq |f(x_0)| > 0$.
Similarly, if $\varphi$ is either elliptic or parabolic with zero hyperbolic step, the closure of the set of recurrent points is exactly the Julia set and the proof is the same as above.
Hence, $F$ is not quasinilpotent.
\end{proof}

\begin{theorem} \label{t:jacobsonradical}
The Jacobson radical of  $A(\mathbb{D})\times_\alpha\mathbb{Z}^+$ is the set $\{F\in A(\mathbb{D})\times_\alpha\mathbb{Z}^+ : \pi_0(F) = 0$, and $\pi_k(F)$ vanishes on the set of recurrent points of $\varphi,~\text{for}~k > 0\}$.
\end{theorem}
\begin{proof}
 Let $I =$ Rad$(C(\mathbb{T})\times_\alpha\mathbb{Z}^+)\bigcap A(\mathbb{D})\times_\alpha\mathbb{Z}^+$. It is clear that $I$ is a left ideal in $\scD$ whose elements are quasinilpotent,
since they are quasinilpotent in $\scC.$ Thus, $I \subset \text{Rad}(\scD).$ By the characterization of the radical of the semicrossed product $\scC$ of \cite{DKM}, if $F \in I$ then $\pi_0(F) = 0$ 
and $\pi_k(F)$ vanishes on the recurrent set of $\vpi$ for $k > 0.$

Conversely, let $F\in$ Rad$(A(\mathbb{D})\times_\alpha\mathbb{Z}^+)$. Let $\pi_i(F) = f_i$ be the $i^{\text{th}}$ Fourier coefficients of $F$. From Lemma~\ref{l:quasinilpotent1},
 $f_0$ must be zero and $f_i$ vanishes on the Denjoy-Wolff point $z_0$ if $z_0\in\mathbb{T}$,  for all $i > 0$. Suppose that not all $f_i$'s vanish on the Julia set. Then there exists $k > 0$ such that $f_{k}(x_0)\neq 0$ for some $x_0$ where $\varphi^{n}(x_0) = x_0$ for some $n$, and $f_i$ vanishes on the set of recurrent points for all $i < k$.
Choose $j$ such that $U^j(U^{k}{f_{k}}) = U^{m}{f_{k}}$ where $m$ is a multiple of $n$. For $l > 0$,
\begin{align*}
||(U^jF)^l|| &\geq ||\pi_{ml}((U^jF)^l)||\\
&\geq |\pi_{ml}((U^jF)^l)(x_0)|\\
&= |f_k(x_0) f_k\circ\varphi^m(x_0)\dots f_k\circ\varphi^{(l-1)m}(x_0)|\\
&= |f_k(x_0)|^l.
\end{align*}
Therefore, $\displaystyle\lim_{l\rightarrow\infty}||(U^jF)^l||^{\frac{1}{l}}\geq |f_k(x_0)| > 0$.
Thus, $U^jF$ is not quasinilpotent. This is a contradiction since $F$ is in the radical.
Hence, $f_i$ vanishes on the set of recurrent points for all $i > 0$.
\end{proof}
 
\begin{theorem}
The Jacobson radical of  $A(\mathbb{D})\times_\alpha\mathbb{Z}^+$ is 
\begin{enumerate}
	\item nonzero if $\varphi$ is hyperbolic or parabolic with positive hyperbolic step.
	\item zero if $\varphi$ is elliptic or parabolic with zero hyperbolic step.
\end{enumerate}
\end{theorem}
\begin{proof}
Note that (2) follows immediately from Theorem~\ref{t:jacobsonradical}.
Let $\varphi$ be any hyperbolic or parabolic with positive hyperbolic step. Then the set of recurrent points of $\varphi$ is not dense in $\mathbb{T}$. Furthermore, this set has Lebesgue measure zero. One can find a function $f\in A(\mathbb{D})$ such that $f$ vanishes exactly on the closure of the set of recurrent points (see \cite{Hoffman} Theorem p.80). Therefore, $Uf$ is in the Jacobson radical of $A(\mathbb{D})\times_\alpha\mathbb{Z}^+$.
\end{proof}

What is more, we not only obtain the same description of the Jacobson radical from \cite{DKM}, but also the following analogue of their theorem.

\begin{corollary}
The semicrossed product $A(\mathbb{D})\times_\alpha\mathbb{Z}^+$ is semi-simple if and only if the recurrent points are dense in $\mathbb{T}$.
\end{corollary}

\cite{HPW} studied semicrossed products of the disk algebra with respect to conformal automorphisms (i.e., M\"{o}bius transformations). There it was shown that it was always the case that the
set of quasinilpotent elements in the semicrossed product coincides with the Jacobson radical.  That contrasts with the situation for semicrossed products $C(X)\times_{\al}\bbZ^+$ where $X$ is an
infinite compact metric space and the endomorphism $\al(f) = f\circ \vpi$ where $\vpi$ is a nontrivial continuous surjection of $X.$  In this case one obtains nilpotent elements not in the Jacobson radical.

For semicrossed products of the disk algebra defined by finite Blaschke products, one can ask if the set of quasinilpotent elements coincides with the radical, as in the conformal case, or whether there are
quasinilpotent elements not in the radical. We begin by answering this question in those cases where the Jacobson radical is nonzero.

\begin{theorem} \label{t:positivehstep}
If $\vpi$ is either hyperbolic or parabolic with positive hyperbolic step, then there exists a nonzero quasinilpotent element that is not in the Jacobson radical.
\end{theorem}
\begin{proof}
Since $\vpi$ is of positive hyperbolic step, the measure of the closure of recurrent points is zero.
By Theorem \ref{t:recpoints}, let $X = \sJ \cup \{z_0\}$ be the closure of recurrent points.
Note that we can find a point $x_0 \in X$ so that $\vpi(x_0) \neq x_0$.
Pick a neighborhood $\sN$ of $x_0$ in $X$ such that $\sN \cap \vpi(\sN) = \emptyset$.
Let $F \in C(X)$ be supported on $\sN$.
Then $F F \circ \vpi = 0$ since $F = 0$ on $\bbT \setminus \sN$ and $F \circ \vpi = 0$ on $\sN$.
By Rudin's Theorem (see \cite{Hoffman} p.81), there exists a function $f \in \sA(\bbD)$ whose restriction to $X$ is $F$. Since $\vpi(X)\subseteq X$, $f f \circ \vpi = F F \circ \vpi = 0$ on $X$.
Consider that $(Uf)^2 = U^2f f \circ \vpi$ is quasinilpotent since $f f \circ \vpi$ vanishes on the set of recurrent points. Hence, $Uf$ is a quasinilpotent element which is not in the radical.
\end{proof}

Next we look at the case where $\vpi$ is elliptic, for which we know the Jacobson radical is zero, and ask if there are nonzero quasinilpotent elements.  First a lemma.

\begin{lemma} \label{l:measurepreserving}
If $\varphi$ is elliptic with the Denjoy-Wolff point $0$, then $\varphi$ is measure-preserving.
\end{lemma}
\begin{proof}
Let $A$ be any measurable set on $\mathbb{T}$, and let $B = \varphi^{-1}(A)$. Then $\varphi(B) = A$. Since $\varphi(0) = 0$, by L$\ddot{\text{o}}$wner's Lemma (see \cite{Pommerenke} p.72), $m(A)\geq m(B)$. Note that $B^c = \varphi^{-1}(A^c)$. Similarly, $m(A^c)\geq m(B^c)$. Therefore, $m(A) + m(A^c)\geq m(B) + m(B^c)$. Since both sides are equal to 1, we have an equality. Thus, $m(A) = m(\varphi^{-1}(A))$.
\end{proof}

\begin{lemma} \label{l:noquasinilpotents}
For any nonzero $f\in A(\mathbb{D})$, if $\varphi$ is elliptic with the Denjoy-Wolff point $0$, then $Uf$ is not a quasinilpotent element of $A(\mathbb{D})\times_\alpha\mathbb{Z}^+$.
\end{lemma}
\begin{proof}
Note that $||(Uf)^n||^{\frac{1}{n}} = \displaystyle\sup_{x\in\mathbb{T}}|f(x) f\circ\varphi(x)\dots f\circ\varphi^{n-1}(x)|^{\frac{1}{n}}$. Since $f\in H^1(\mathbb{T}), \log |f|$ is integrable (see \cite{Hoffman} Corollary p.52). By Lemma~\ref{l:measurepreserving} and the Ergodic Theorem, $\log |f(x) f\circ\varphi(x)\dots f\circ\varphi^{n-1}(x)|^{\frac{1}{n}} = \frac{1}{n}\displaystyle\sum^{n-1}_{k = 0}\log |f\circ\varphi^k(x)|$ converges to $\int_\mathbb{T}\log |f| dm$ almost everywhere. Let $x_0\in\mathbb{T}$ be such that the above convergence holds. Then 
\begin{align*}
\displaystyle\lim_{n\rightarrow\infty}||(Uf)^n||^{\frac{1}{n}} &\geq \displaystyle\lim_{n\rightarrow\infty}|f(x_0) f\circ\varphi(x_0)\dots f\circ\varphi^{n-1}(x_0)|^{\frac{1}{n}}\\
&= \exp\{\displaystyle\lim_{n\rightarrow\infty}\log |f(x_0) f\circ\varphi(x_0)\dots f\circ\varphi^{n-1}(x_0)|^{\frac{1}{n}}\}\\
&= \exp\{\int_\mathbb{T}\log |f| dm\}\\
&> 0.
\end{align*}
Hence, $Uf$ is not a quasinilpotent element.
\end{proof}

\begin{theorem} \label{t:elliptic}
If $\varphi$ is elliptic, then $A(\mathbb{D})\times_\alpha\mathbb{Z}^+$ has no nonzero quasinilpotent elements.
\end{theorem}
\begin{proof}
Any elliptic map is conjugate to an elliptic map which fixes zero.
Then by Theorem~\ref{t:conjugate}, the semicrossed product is completely isometrically isomorphic to one for which the map $\vpi$ is elliptic with Denjoy-Wolff point $0.$ Thus there is no loss of
generality in assuming this is the case.

Let $F\in A(\mathbb{D})\times_\alpha\mathbb{Z}^+$ be any quasinilpotent element. Let $\pi_i(F) = f_i$ be the $i^{\text{th}}$ Fourier coefficient of $F$. If $f_0$ is not identically zero, then $F$ is not a quasinilpotent element. Assume that there exists $k > 0$ such that $f_{k}\neq 0$ where $f_i = 0$ for all $i < k$. For $n > 0$,
\begin{align*}
||F^n||^{\frac{1}{n}} &\geq ||\pi_{nk}(F^n)||^{\frac{1}{n}}\\
&= \displaystyle\sup_{x\in\mathbb{T}} |f_k(x) f_k\circ\varphi^k(x)\dots f_k\circ\varphi^{(n-1)k}(x)|^{\frac{1}{n}}\\
&= ||(U_{\be}{f_k})^n||^{\frac{1}{n}},
\end{align*}
where $U_{\be}{f_k}\in A(\mathbb{D})\times_\beta\mathbb{Z}^+$ and $\beta(f) = f\circ\varphi^k$.
Clearly, $\varphi^k$ is also elliptic with the Denjoy-Wolff point $0$.
By Lemma \ref{l:noquasinilpotents}, $\displaystyle\lim_{n\rightarrow\infty}||F^n||^{\frac{1}{n}}\geq\displaystyle\lim_{n\rightarrow\infty}||(U_{\be}{f_k})^n||^{\frac{1}{n}} > 0$.
Then, $F$ is not a quasinilpotent element. This is a contradiction. Hence, $F$ is zero.
\end{proof}

The question of whether the Jacobson radical coincides with the set of quasinilpotent elements has resolved in all cases except where the map $\vpi$ is parabolic with zero hyperbolic step. The argument
in the elliptic case cannot be used here, since the map is not measure-preserving.  This question remains open.

\bibliographystyle{plain}

\end{document}